\documentclass{amsart}
\usepackage{amsmath,amsthm,amsfonts,amssymb,latexsym,mathrsfs,graphicx,tikz,array}
\usepackage{subcaption}
\usepackage{hyperref}

\usepackage{xpatch}
\makeatletter
\xpatchcmd{\@thm}{\thm@headpunct{.}}{\thm@headpunct{}}{}{}
\makeatother
\usepackage{enumerate}
\usepackage[shortlabels]{enumitem}
\usepackage{color}

\newcommand{\cvd}{\hspace*{\fill}
	{\rm \hbox{\vrule height 0.2 cm width 0.2cm}}}
\renewcommand{\qed}{\cvd}
\newcolumntype{P}[1]{>{\centering\arraybackslash}p{#1}}

\newtheorem{lemma}{Lemma}[section]
\newtheorem{theorem}[lemma]{Theorem}
\newtheorem{proposition}[lemma]{Proposition}
\newtheorem{corollary}[lemma]{Corollary}
\renewenvironment{proof}[1][\noindent \proofname]{{\sc #1. }}{\qed}
\theoremstyle{definition}

\newtheorem{example}[lemma]{Example}
\newtheorem{definition}[lemma]{Definition}
\newtheorem{notation}[lemma]{Notation}{\bf}{\rm}

\newcommand{\abs}[1]{\ensuremath{\left| #1 \right|}}
\newcommand{\op}{\operatorname}

\newcommand{\irr}[1]{{\op{Irr}}(#1)}
\newcommand{\Min}[1]{\op{Min}_G(#1)}
\newcommand{\Lin}[1]{\op{Lin}_G(#1)}
\newcommand{\Mcd}[1]{\op{Mcd}_G(#1)}
\newcommand{\Lcd}[1]{\op{Lcd}_G(#1)}

\begin{document}

\title{\bf Some properties of normal subgroups determined from character tables}

\author[Z. Akhlaghi]{Z. Akhlaghi$^{1,2}$}
	\author [M. J.  Felipe ]{ M. J.  Felipe $^3$}
	\author[ M .K. Jean-Philippe]{ M. K. Jean-Philippe$^{4}$}

\address{$^{1}$ Faculty of Mathematics and Computer Science, Amirkabir University of Technology (Tehran Polytechnic), 15914 Tehran, Iran.}
	\address{$^{2}$ School of Mathematics,
		Institute for Research in Fundamental Science (IPM)
		P.O. Box:19395-5746, Tehran, Iran. }
\address{$^{3}$  Instituto Universitario de Matemática Pura y Aplicada (IUMPA-UPV), Universitat Polit\`ecnica de Val\`encia, Valencia, Spain. ORCID: 0000-0002-6699-313.}
\address{$^{4}$  Departamento de Matemáticas, Instituto Tecnológico de Santo Domingo (INTEC) and Universidad Autónoma de Santo Domingo (UASD), Santo Domingo, República Dominicana.}
\email{ \newline \text{(Z. Akhlaghi) }z\_akhlaghi@aut.ac.ir
\newline \text{(M. J. Felipe) }mfelipe@mat.upv.es
\newline\text{ (M. K. Jean-Philippe) }1097587@est.intec.edu.do, mjean46@uasd.edu.do}

\thanks{The first author is supported by a grant from IPM (No. 1402200112) and the second author is supported by Proyecto CIAICO/2021/163, Generalitat Valenciana (Spain).
}

\date{}

\maketitle

\begin{abstract}
\noindent $G$-character tables of a finite group $G$ were defined in \cite{FPRS}. These tables can be very useful to obtain certain structural information of a normal subgroup from the character table of $G$. We analyze certain structural properties of normal subgroups which can be determined using their $G$-character tables. For instance, we prove an extension of the Thompson's theorem from minimal $G$-invariant characters of a normal subgroup. We also obtain a variation of Taketa's theorem for hypercentral normal subgroups considering their minimal $G$-invariant characters. This generalization allows us to introduce a new class of nilpotent groups, the class of $nMI$-groups, whose members verify that its nilpotency class is bounded by the number of irreducible character degrees of the group. 
\medskip



\end{abstract}


\section{Introduction}

In general, if $N$ is a normal subgroup of a finite group $G$, the character table of $N$ can not be determined from the character table of $G$. In \cite{FPRS}, the $G$-character table of $N$ obtained from the character table of $G$  is defined (see Section 2). They are non-singular matrices that provide certain arithmetical relations among significant integers associated to the irreducible characters of $N$ such as their character degrees, indices of inertia subgroups and ramification numbers. The character table of $G$ determines those irreducible characters of $G$ that are over the same minimal $G$-invariant character of $N$, which are defined as the sum of  characters of $N$
 belonging in the same orbit determined by the conjugation action of $G$ on the set of irreducible characters of $N$. These characters of $N$ form a basis of the $\mathbb{C}$-vector space of the $G$-invariant class functions of $N$, and they become relevant to $G$-character tables of $N$, playing a similar role that irreducible characters do to the character table of a group $G$.

The nilpotency of a normal subgroup $N$ of a group $G$ can be recognized from the character table $G$, but the nilpotency class of $N$ cannot be recognized. In fact, we can not deduce if a normal subgroup $N$ is abelian (see \cite{Doc} and \cite{Sak}). However, the commutator $[G,N]$ can be computed from $G$-character tables of $N$ as  the intersection of kernels of linear $G$-invariant characters of $N$ and this information can be read from $G$-character tables of $N$ (Corollary~\ref{commutator2}). In particular, we can read from the character table of $G$ if a normal subgroup $N$ is hypercentral and we can obtain its lower central $G$-series. 

We can not compute prime divisors of irreducible character degrees of $N$ from the character table of $G$. However,  prime divisors of the minimal $G$-invariant character degrees of a normal subgroup $N$ are known, considering some relations that arise from $G$-character tables of $N$ (Corollary~\ref{relations}). We present a generalization of the well-known Thompson's theorem (see 12.2 in \cite{ISA}) applied to minimal $G$-invariant character degrees of $N$ (Theorem~\ref{Thompson}). This generalization gives us a sufficient condition to deduce whether a normal subgroup $N$ has a normal $p$-complement from the character table of $G$.

Finally, we present a new class of groups which are determined by the following property: for every irreducible character $\chi \in \irr{G}$ there exists a normal subgroup $H$ of $G$ such that $\{\chi\}=\irr{G|\lambda }$, with $\lambda$ a linear $G$-invariant character of $H$. These groups are called $nMI$-groups. A variation of Taketa's theorem (Theorem~\ref{normal}) (see 5.13 in \cite{ISA}), considering minimal $G$-invariant character degrees of certain hypercentral normal subgroups of $G$, allows us to prove that if $G$ is a $nMI$-group, then $G$ is nilpotent and the nilpotency class of $G$ is bounded by the number of distinct irreducible character degrees of $G$ (Theorem~\ref{invariant}). Moreover, if $G$ is a $nMI$-group, all normal subgroups of $G$ are $G$-invariant $nMI$-subgroups (Definition~\ref{normal-invariant}).

In the sequel, all groups considered are finite. The notation and terminology used are standard, and they are taken mainly from the books \cite{ISA}, \cite{Rose} and \cite{Hu}.

\section{Preliminaries}\label{section 2}

The group $G$ acts by conjugation on the set $\irr{N}$ of irreducible complex characters of $N$.  If $\theta\in \irr{N}$, then
$\theta^g\in \irr{N}$ and $\widehat{\theta}=\sum_{i=1}^t \theta^{g_i}$ is a {\it minimal $G$-invariant character} of $N$, where $\{\theta^{g_i}\mid i=1,\ldots, t\}=\{\theta^{g}\mid g\in G\}$ is the orbit of $\theta$ under the action by conjugation of $G$ on $\irr{N}$, while the set $\{g_i\in G\mid i=1,\ldots, t\}$  is a right transversal in $G$ of $I_G(\theta)=\{g\in G\mid \theta^g=\theta\}$, the {\it inertia subgroup} of $\theta$ in $G$. Clifford's theorem (see \cite[Theorem~(6.2)]{ISA}) states that irreducible characters $\chi$ of $G$ whose restrictions $\chi_{N}$ to $N$ have $\theta$ as constituent satisfy $\chi_N= e\,\widehat{\theta}$ for suitable integers $e$, known as {\it ramification numbers}. This fact introduces an equivalence relation (with respect to $N$) on the set $\irr{G}$ of irreducible complex characters of $G$, being two elements equivalent if their restrictions to the normal subgroup $N$ have a common irreducible constituent. 

We recall the main notation used in Notation 4.5 of \cite{FPRS} about character G-tables.

\begin{notation}\label{notation_tables}
Let $N$ be a normal subgroup of a group $G$. Denote:
\begin{itemize}\item  $\{n_1^G, \ldots, n_k^G\}$  the set of $G$-conjugacy classes of $N$;
 \item $\textsf{D}=(d_{ij})$ a diagonal matrix with entries $d_{ij}=\delta_{ij}|n_i^G|$, where $\delta_{ij}$ is the Kronecker delta function, for $i,j\in \{1, \ldots, k\}$;
\item $\Delta = \{\chi_1=1_G, \chi_2, \ldots, \chi_k\}$ a representative system of the equivalence classes in the equivalence relation with  respect to $N$, defined on $\irr{G}$;
\item $\Omega= \{\theta_i \in \irr{N} \mid 1 \leq i \leq k\}$ such that $\chi_i \in \irr{G|\theta_i}$, for each $i \in \{1, \ldots, k\}$;
\item $t_i=|G:I_{G}(\theta_i)|$, $e_{i} = [(\chi_i)_N, \theta_i]\neq 0$, $1\le i\le k$.
\item $\textsf{X}=(x_{ij}) \in \textup{M}_k(\mathbb{C})$, with entries $x_{ij}=\chi_i(n_j)$ for $1\leq i,j \leq k$, is the $G$-character table of $N$ constructed from $\Delta$, and $\overline{\textsf{X}}^t$ its transposed conjugate matrix; i.e if $\textsf{X}=(x_{ij})$, then $\overline{\textsf{X}}^t=(y_{ij})$ being $y_{ij}=\overline{x_{ji}}$ the complex conjugate of $x_{ji}$, for $1\le i,j\le k$.
\item  $\Lambda_{\textsf{X}}=\text{diag}(\lambda_1,\ldots,\lambda_k)$ the  diagonal matrix with entries $\lambda_i=\abs{N}t_ie_i^2$, $ 1\le i\le k.$
\end{itemize}
\end{notation}

\begin{theorem}\label{square}
With  Notation~\ref{notation_tables}, it holds that  \textup{$\Lambda_{\textsf{X}}=\textsf{X}\textsf{D}\overline{\textsf{X}}^t.$} In particular, \textup{$\textsf{X}$} is non-singular.
\end{theorem}

As a consequence of   Theorem~\ref{square}, it is obtained that the equivalence  classes of irreducible characters of $G$ are in one-to-one correspondence with the orbits of irreducible characters of $N$ under the action of $G$ by conjugation, by Clifford's theorem,  and we can recognize the sets $\irr{G|\theta}$, for every $\theta \in \irr{N}$, from the character table of $G$. Moreover, the following arithmetical relations between the aforementioned integers $t_i$, $e_{i}$ and $\theta_i(1)$.

\begin{corollary}\label{relations}
 With Notation~\ref{notation_tables}, the next integer relations hold, where the corresponding right sides can be computed from the  $G$-character table $\textsf{X}$ of $N$:
\begin{align}
e_{i}^2t_i&=\dfrac{\lambda_i}{\abs{N}} \tag{$A_i$} \\
t_i\theta_i(1)^2&= \dfrac{|N|\chi_i(1)^2}{\lambda_i} \tag{$B_i$}
\end{align}
for $1\le i\le k$.

\end{corollary}

\section{Main results}\label{section 3}

In general, the nilpotency and solvability of a normal subgroup $N$ of a group $G$ can be obtained from the character table of $G$. However, it is not possible to know the derived length when $N$ is solvable or the nilpotency class when $N$ is nilpotent. In fact, we are not able to  find out if $N$ is abelian by the character table of $G$.  (It is  well-known that Problem 10 of Brauer's famous list in \cite{B} has a negative answer). 

\begin{proposition}
Let $N$ be a normal subgroup of $G$. The character table of $G$ determines if $N$ is nilpotent or solvable.
\end{proposition}

\begin{proof} Let $p$ be a prime divisor of the order of $N$. By Higman's theorem (see 8.21 of \cite{ISA}), the character table of $G$ determines the sets of  prime divisors  of orders of elements of $G$. In fact, an element $x \in N$ is a $p$-element if and only if
$$\chi(x)^{|G|_p} \equiv \chi(1) \text{  (mod $p$)},$$
for all $\chi \in \irr{G}$ (see Corollary 5.3 of \cite{Bel}). Therefore, from the character table of $G$ we can realize which $G$-conjugacy classes of $N$ contain elements with  $p$-power order. Let $P$ be a Sylow $p$-subgroup of $N$ and
$T= \dot{\bigcup}_{x \in P}^{}x^G,$
the disjoint union of $G$-conjugacy clases of $p$-elements of $N$. The $G$-conjugacy class size $|x^G|$, for $x \in P$, can be obtained from the character table of $G$. Then $P$ is normal in $N$ if and only if $|T|+1=|P|$.  Hence, the character table of $G$ allows us to know the nilpotency of $N$.

We observe that $N$ is solvable if and only if there exists a series
$$1=N_0 \unlhd N_1 \unlhd N_2 \unlhd \ldots \unlhd N_r=N$$
such that $N_i \unlhd G$, for $i \in \{0, \ldots ,r\}$, and $N_i/N_{i-1}$ is nilpotent, for $i \in \{1, \ldots, r\}$. Since we can check if the quotient groups $N_i/N_{i-1}$ are nilpotent from the character table of $G/N_{i-1}$ (obtained from the character table of $G$), then we can deduce if $N$ is solvable or not from the character table of $G$.
\end{proof}

We consider the following set of irreducible characters of $N$
$$\Lin{N}=\{\theta \in \irr{N} \text{ }| \text{ } \theta \text{ is a linear } G\text{-invariant character of } N \}.  $$

\begin{proposition}\label{commutator}
 Let $N$ be a normal subgroup of a group $G$, $\theta \in \irr{N}$ and $\chi \in \irr{G|\theta}$. Then $[N,G] \subseteq 
\ker(\chi)$ if and only if $\theta \in \Lin{N}$.  
In particular, $|N:[N,G]| = | \Lin{N}|$.

\end{proposition}

\begin{proof} Firstly, we prove that
$ \Lin{N}=\{\theta \in \irr{N} \text{ }| \text{ }  [G,N] \subseteq \ker(\theta)\}.$
Let $[n,g]$ be a generator of $[N,G]=[G,N]$, with $g \in G$ and $n \in N$. Let $\theta \in \Lin{N}$, then
$$\theta([g,n])=\theta(g^{-1}n^{-1}gn)=\theta^{g^{-1}}(n^{-1})\theta(n)=\theta(n^{-1})\theta(n)=\theta(n^{-1}n)=\theta(1)=1.$$
Then $[N,G] \subseteq \ker(\theta).$

Reciprocally, if $\theta \in \irr{N}$ and $[N,G] \subseteq \ker(\theta)$, then $\theta$ is a linear character of $N$ since $N' \subseteq \ker(\theta)$. Moreover, if $g \in G$ and $n \in N$, then
$$1=\theta(1)=\theta(g^{-1}n^{-1}gn)=\theta(g^{-1}n^{-1}g)\theta(n).$$
Therefore
$$\theta^{g^{-1}}(n^{-1})\theta(n)=1$$ for every $n \in N$ and $g \in G$. Then
$\theta^g(n)=\theta(n),$
for every $n\in N$ and $g \in G$ and $\theta$ is a $G$-invariant character of $N$; that is $\theta \in \Lin{N}$. 
In particular, if $[N,G]  \subseteq \ker(\chi)$ for some $\chi \in \irr{G|\theta}$ with $\theta \in \irr{N}$, then $[G,N] \subset \ker(\theta)$ and $\theta \in \Lin{N}$.

Moreover,
$$\irr{N/[G,N]}=\{\theta \in \irr{N} \mid  [G,N] \subseteq \ker(\theta)\} =  \Lin{N}$$
and  $|N:[N,G]| = | \Lin{N}|$, since $N/[G,N]$ is abelian.
\end{proof}

\begin{corollary} \label{commutator2} Let $N$ be a normal subgroup of a group $G$, then the subgroup $[N,G]$ and the sets $\irr{G|\theta}$, with $\theta \in \Lin{N}$,  are determined from the character table of $G$.

\end{corollary} 

\begin{proof} Let $\Delta$ be a representative system of the equivalence classes in the equivalence relation with  respect to $N$ defined on $\irr{G}$. By Corollary~\ref{relations}$(B_i)$,  we have $\theta \in \Lin{N}$ if and only if
$$\lambda_i=|N|\chi_i(1)^2$$
where $\lambda_i$ is an element of the main diagonal of the matrix $\Lambda_X$ obtained from a $G$-character table $X$ of $N$, with $\chi_i \in \Delta$ such that $[(\chi_i)_N,\theta] \neq 0$. As a consequence, we can read from the correspondent $G$-character table $X$ of $N$ if a constituent of $(\chi_i)_N$ is a linear $G$-invariant character of $N$. Observe that
$$[N,G] = \bigcap_{\theta \in \Lin{N}}
\ker(\theta) =  \bigcap_{
\begin{scriptsize}
\begin{array}{c}

 \theta \in \Lin{N} \\
 \chi \in \irr{G|\theta} \cap \Delta
\end{array}
\end{scriptsize}} \ker(\chi) \cap N.$$
Therefore, the subgroup $[N,G]$ can also be determined from the character table of $G$.
\end{proof}

Next, we introduce some notation and some elementary results on hypercentral subgroups that we will use later. Let $N$ be a normal subgroup of $G$. A $G$-series of $N$ is a normal series of $N$
$$1=N_0 \trianglelefteq N_1 \trianglelefteq \cdots \trianglelefteq N_s=N$$
with $N_{i} \trianglelefteq G$, for $i \in\{1, \ldots, s\}$. In particular, we said that a $G$-series of $N$ is a chief $G$-series if $N_{i}/N_{i-1}$ is a chief factor of $G$, for every $i \in\{1, \ldots, s\}$.

\begin{definition} Let $N$ be a normal subgroup of a group $G$. The descending $G$-series
$$\Gamma_G^1(N)=N \trianglerighteq \Gamma_G^2(N)=[\Gamma_G^1(N),G] \trianglerighteq \Gamma_G^3(N)=[\Gamma_G^2(N),G] \trianglerighteq \ldots$$
is called  lower central $G$-series of $N$.

\end{definition}

If there exists some integer $r$ such that $\Gamma_G^r(N)=1$, then $N$ is a hypercentral subgroup of $G$. As a consequence of Corollary~\ref{commutator2} we obtain the following.


\begin{corollary}\label{Glength} Let $N$ be a normal subgroup of a group $G$, then the character table of $G$ determines the $G$-hypercentre of $N$. In particular, if $N$ is a hypercentral subgroup of $G$, the character table of $G$ determines  lower central $G$-series of $N$. 

\end{corollary}



In the literature, there exist  different equivalent descriptions and several results relating to the hypercenter of a group (see, for instance, \cite{baer}, Section 6 in IV of \cite{DH}, \cite{peng1}, \cite{peng2} and Section 6 in Appendix C of \cite{Wei}). With notation in \cite{peng2}, we consider the hypercenter $H(G)$ of $G$ as the upper ascending central series of $G$.

\begin{definition}\label{Hypercenter}

We consider the ascending $G$-series of a normal subgroup $N$:
$$1=\textbf{Z}_G^{0}(N) \trianglelefteq \textbf{Z}_G^{1}(N)=\textbf{Z}(G) \cap N \trianglelefteq \textbf{Z}_G^{2}(N) \trianglelefteq \ldots $$
such that $$\textbf{Z}_G^{i+1}(N)/\textbf{Z}_G^{i}(N)= \textbf{Z}(G/\textbf{Z}_G^{i}(N)) \cap N/\textbf{Z}_G^{i}(N),$$ for $i=1,2, \ldots$. This series is called the  upper central $G$-series of $N$. 
We define  the $G$-hypercenter of $N$ as $H_G(N)=\bigcup \limits_{i} \textbf{Z}_G^{i}(N)$. We have that $H(G) \cap N = H_G(N) \subseteq H(N)$.

\end{definition}

The following result is elementary and can be proven by induction easily.

\begin{proposition} \label{lenght} Let $N$ be a subgroup of $G$. If there exists an integer $r \geq 1$ such that $\Gamma^r_G(N)=1$ in the lower central $G$-series of $N$, then $\textbf{Z}_G^{r-1}(N)=N$ in the upper central $G$-series of $N$. Reciprocally, if $\textbf{Z}_G^{r}(N)=N$, for some integer $r \geq 1$, then $\Gamma^{r+1}_G(N)=1$. 
\end{proposition}




The above proposition allows us to define the concept of the hypercentral $G$-length of a normal subgroup of a group $G$.

\begin{definition} Let $N$ be a hypercentral normal subgroup of a finite group $G$. The least integer $k$ such that $\Gamma^{k+1}_G(N)=1$ is equal to the least integer $k$ such that $\textbf{Z}_G^{k}(N)=N$ and $k$ is called \textbf{the hypercentral $G$-length of $N$}. We denote this integer by $\textbf{l}_G(N)$.

\end{definition}

 In \cite{garr}, it is proved that the Frattini subgroup of a solvable group can be determined from the character table; however, this  fact does not happen when the group is not solvable. If $N$ is a hypercentral normal subgroup in $G$, then $N$ is nilpotent and  $[N,N] \leq \Phi(N) \leq \Phi(G) \cap N$, where $\Phi(G)$ is the Frattini subgroup of $G$, but the subgroups $[N,N]$ and $\phi(N)$ can not be read from the character table of $G$.  Using induction on the hypercentral $G$-length of hypercentral normal subgroups of a group $G$, we obtain the following result. 

\begin{corollary} \label{fratt}
Let $N$ be a hypercentral subgroup in $G$. Then $[G,N] \leq \Phi(G) \cap N$. 

\end{corollary}

\begin{proof} We argue by induction on the hypercentral $G$-length of $N$, which is donated by $r$. If $r=1$, then $[G,N]=1$ and the result is trivial. Suppose that $N_1=[G, N] \neq 1$. Since $N$ is a hypercentral subgroup in $G$, then $N_1 < N$ and $N_1$ is a hypercentral subgroup in $G$. The hypercentral $G$-length of $N_1$ is less  than or equal to $r-1$. By induction, $[G,N_1] \leq \Phi(G) \cap N_1 \leq \Phi(G) \cap N.$ 

Let $M$ be a maximal subgroup of $G$, we have $[G,N_1] \leq M \cap N$. We claim that $N_1 \leq M$. Assume that $G=N_1M$, then $N=N_1(M \cap N)$ and $[G,N]=[G,N_1][G, M \cap N]$. Since $M \cap N$ is normal in $M$, we have  $[M,M \cap N]$ is a subgroup of $M \cap N$. Moreover, $[N_1,M \cap N] \leq [N_1,G] \leq M \cap N$. Therefore, $[G, M \cap N]=[M, M \cap N][N_1,M \cap N] \leq M \cap N$. As a consequence, $N_1=[G,N]= [G,N_1][G, M \cap N] \leq M \cap N$, a contradiction, and the proof is complete.
\end{proof}


We point out that it is not difficult to find examples that show that the reverse implication of the above  result is not true (for instance, $G=D_{18}$ and $N= \Phi(G)$).

Now, we denote by $$\Min{N}=\{\widehat{\theta}\mid \theta \in \irr{N}\}=\{\widehat{\theta}_i\mid i=1,\dots,k\}$$ the set of minimal $G$-invariant characters of $N$. By Corollary~\ref{relations}$(B_i)$, we have $$ \widehat{\theta_i}(1)\theta_i(1)= \dfrac{|N|\chi_i(1)^2}{\lambda_i},$$
which is  obtained from a $G$-character table $X$ of $N$. Then, prime divisors of $\widehat{\theta}(1)$ are known from this equation, although  we can not  find out the  prime divisors of $\theta(1)$. In particular, if a prime $p$ does not divide $\widehat{\theta}(1)$ for every minimal $G$-invariant character $\widehat{\theta}$ of $N$, then we can deduce that $N$ has an abelian normal  Sylow $p$-subgroup by It\^o-Michler's theorem (see Corollary 4.11 of \cite{FPRS}).\\

However, we consider the opposite situation. Assume there exists a prime $p$ divides $\widehat{\theta}(1)$ for every $\widehat{\theta}\in \Min{N} \setminus \Lin{N}$. We observe that this fact does not imply that $p$ divides every irreducible character degree of $N$. In this situation, we obtain the following  generalization of the well-known Thompson's theorem for minimal $G$-invariant characters of $N$.  

We keep the notation of Theorem 12.1 of \cite{ISA}. Let
$$\mathscr{S}(N)=\{\theta \in \irr{N} \text{ } | \text{ } p \nmid \theta(1) \text{ and } p \nmid o(\theta)\}.$$
and $s(N)=\sum_{\theta \in \mathscr{S}(N)}\theta(1)^2$. By Theorem 12.1 of \cite{ISA}, we have
$$|{\bf O}^p (N)| \equiv s(N) \text{  } (\text{mod }p ).$$

We consider the set $\mathscr{S}_G(N)=\{ \theta \in \Lin{N} |\text{ } p \nmid o(\theta)\} \subseteq \mathscr{S}(N).$

\begin{theorem}\label{Thompson}
Let $p$ be a prime integer. Suppose $p$ divides $\widehat{\theta}(1)$ for every $\widehat{\theta}\in \Min{N} \setminus \Lin{N}$, then $N$ has a normal $p$-complement.
\end{theorem}

\begin{proof} We have
$$s(N)=\sum_{\theta \in \mathscr{S}(N)}\theta(1)^2 = \sum_{\theta \in \mathscr{S}_G(N)}\theta(1)^2+\sum_{\theta \in \mathscr{S}(N) \setminus \mathscr{S}_G(N)}\theta(1)^2 .$$

If $\theta \in \mathscr{S}(N) \setminus \mathscr{S}_G(N)$, then $p$ divides the index of the inertia group $I_G(\theta)$ of $\theta$ in $G$. Moreover, every $G$-conjugate character of $\theta$ is also in the set $\mathscr{S}(N) \setminus \mathscr{S}_G(N)$. Therefore,
$$\sum_{\theta \in \mathscr{S}(N) \setminus \mathscr{S}_G(N)}\theta(1)^2 \equiv 0 \text{  } (\text{mod }p )$$
and then
$$s(N) \equiv \sum_{\theta \in \mathscr{S}_G(N)}\theta(1)^2\text{  } (\text{mod }p). $$

If $\theta \in \mathscr{S}_G(N)$, then $\theta \in \Lin{N}$ and $p \nmid o(\theta)=|N:\ker(\theta)|$. By Proposition~\ref{commutator} we have $[G,N] \subseteq \ker(\theta)$. Then $[G,N]{\bf O}^{p'}(N) \subseteq \ker(\theta)$ and $\theta \in \irr{N/[G,N]{\bf O}^{p'}(N)}$.
Reciprocally, if $\theta \in \irr{N/[G,N]{\bf O}^{p'}(N)}$, then $N' \subseteq [G,N] \subseteq \ker(\theta)$ and ${\bf O}^{p'}(N) \subseteq \ker
(\theta)$. We conclude that $\theta \in \Lin{N}$ and $p \nmid o(\theta)$, and then $\theta \in \mathscr{S}_G(N)$. Therefore,
$$\sum_{\theta \in \mathscr{S}_G(N)}\theta(1)^2 = |N: [G,N]{\bf O}^{p'}(N)|$$
and this integer is a $p'$-number. By Theorem 12.1 of \cite{ISA},
$$|{\bf O}^{p}(N)| \equiv s(N) \text{  } (\text{mod }p) \equiv \sum_{\theta \in \mathscr{S}_G(N)}\theta(1)^2  \text{  } (\text{mod }p). $$
Therefore, $p$ does not divide $|{\bf O}^p(N)|$ and ${\bf O}^p(N)$ is a normal $p$-complement of $N$.
\end{proof}

\begin{example} It is not difficult to find examples satisfying the previous theorem. For instance, using GAP (see \cite{GAP}), the group $G=SmallGroup(120,3)$ has a cyclic normal subgroup $N$ of order 20 and we can check from Corollary~\ref{relations}$(B_i)$, that $p=2$ divides $\widehat{\theta}(1)$ for every $\widehat{\theta}\in \Min{N} \setminus \Lin{N}$. Then, we deduce that $N$ has a normal $2$-complement by Theorem~\ref{Thompson}. In fact, the unique prime number dividing a degree of a minimal $G$-invariant character of $N$ is $2$. Therefore, a minimal $G$-invariant character of $N$  has either degree 1 (which means it is a linear $G$-invariant character of $N$) or  $2$-power degree. However, in relation to Thompson's theorem, $p=2$ does not divide every irreducible character degree of $N$ since $N$ is abelian. 

\end{example}



In \cite{LE}, using the classification of nonabelian simple groups, it is proved that $MI$-groups are solvable. A group $G$ is said to be a $MI$-group if every nonlinear irreducible character of $G$ is a multiply imprimitive character or m.i. character for short (that is, for every irreducible character  $\chi \in \irr{G}$ there exists  a proper subgroup $U$ of $G$ and an irreducible character $\lambda \in \irr{U}$ such that $\lambda^G=m\chi$ for some nonnegative integer $m$). We introduce the following definition. 

\begin{definition} Let $G$ be a finite group and $\chi \in \irr{G}$. We say that $(H_{\chi},\lambda_{\chi})$ is a {\bf linear character pair  with respect to $\chi$} if $H_{\chi}$ is a normal subgroup of $G$, $\lambda_{\chi} \in \Lin{H_{\chi}}$ and $\irr{G|\lambda_{\chi}}=\{\chi\}$ (so $\chi$ is m.i. character with notation in \cite{LE}). By Problem (6.3) of \cite{ISA}, we have that $\chi$ and $\lambda_{\chi}$ are fully ramified with respect to $G/H_{\chi}$. 
\end{definition}

Recall that a group $G$ is a monomial group or $M$-group if for every irreducible character $\chi \in \irr{G}$ there exists a subgroup $H$ and a linear irreducible character $\varphi \in \irr{H}$ such that $\chi = \varphi^G$. By Taketa's theorem  (see Theorems 5.12 and 5.13 in \cite{ISA}), every $M$-group is solvable and the derived length of $G$ is bounded by the number of irreducible character degrees of $G$. Structural properties of $nM$-groups were introduced in \cite{how} as a generalization of $M$-groups. A group $G$ is said to be a $nM$-group if for every irreducible character $\chi \in \irr{G}$ there exists a normal subgroup $H$ and a linear irreducible character $\varphi \in \irr{H}$ such that $\chi = \varphi^G$. Observe that in this situation, we have $\chi(1)=|G:H|$ and, using the Frobenius reciprocity,  the ramification number of $\chi$ with  respect to $N$ is equal to $1$. 

Now, we consider the following class of groups as a variation of $nM$-groups and $MI$-groups. 

\begin{definition} \label{normal-invariant}
A group $G$ is said to be a {\bf $nMI$-group} if for every irreducible character $\chi \in \irr{G}$ there exists $(H_{\chi},\lambda_{\chi})$ a linear character pair with respect to $\chi$.

\end{definition}

Trivially, every abelian group $G$ is an $nMI$-group onsidering $\lambda_{\chi}^G= \chi$ for every irreducible character $\chi \in \irr{G}$. Moreover, the character table determines if a group $G$ is a $nMI$-group.

\begin{corollary}
 A finite group $G$ is $nMI$-group if and only if for every $\chi \in \irr{G}$ there exists a normal subgroup $N$ of $G$ such that $[G,N] \subseteq \ker(\chi)$ and the respective element  of the principal diagonal $\Lambda_X$ from a $G$-character table $X$ of $N$ associated to $\chi$ is equal to $|G|$. 
\end{corollary}

\begin{proof} It is a consequence of Corollary~\ref{relations}$(B_i)$ and Proposition~\ref{commutator}.
\end{proof}

In a more general context, we present the following definitions for normal subgroups.

\begin{definition} Let $N$ be a normal subgroup of a group $G$ and $\theta \in \irr{N}$. We say that $(H_{\theta},\lambda_{\theta})$ is a {\bf linear $G$-character pair with  respect to $\theta$} if $H_{\theta}$ is a normal subgroup of $G$ contained in $N$, $\lambda_{\theta} \in \Lin{H_{\theta}}$ and $\irr{G|\theta}=\irr{G|\lambda_{\theta} }$. 
\end{definition}

Notice that the last condition above is equivalent to $(\lambda_{\theta})^N= m \widehat{\theta}$ (and then the restriction $\theta_{H_{\theta}}=m\lambda_{\theta}$) for some positive interger $m$.

\begin{definition} \label{normal_t}
A normal subgroup $N$ of a group $G$ is said to be a {\bf $G$-invariant $nMI$-subgroup} if for every $\theta \in \irr{N}$, there exists $(H_{\theta},\lambda_{\theta})$ a linear $G$-invariant character pair with respect to $\theta$. In particular, when $N=G$ we have that $G$ is a $nMI$-group.
\end{definition}

We remark that if $(H_{\theta},\lambda_{\theta})$ is a linear $G$-invariant character pair with respect to $\theta$, then not necessarily $(H_{\theta},\lambda_{\theta})$ is a linear character pair with respect to $\theta$ (see, for instance, Example 3.22)

\begin{corollary} The character table of $G$ determines if a normal subgroup $N$ of $G$ is a $G$-invariant $nMI$-subgroup. In this case, if $\chi \in \irr{G | \theta}$, with $\theta \in \irr{N}$  and $(H_{\theta}, \lambda_{\theta})$ is a linear $G$-character pair respect to $\theta$, then $[G,H_{\theta}] \subseteq \ker(\chi)$ and the respective element of the principal diagonal $\Lambda_X$ from a $G$-character table $X$ of $N$ associated to $\chi$ is equal to $|H_{\theta}|\chi(1)^2$. 

\end{corollary}

\begin{proof} We obtain the result by using Corollary~\ref{relations}$(B_i)$ and Proposition~\ref{commutator} and the fact that the equivalence classes in the equivalence relation with  respect to $N$, defined on $\irr{G}$, are detected in the character table of the group $G$ (Corollary 4.6 of \cite{FPRS}).
\end{proof}

Let
$$ \Min{N}=\{\widehat{\theta}\mid \theta \in \irr{N}\}$$ be the set of minimal $G$-invariant characters of $N$.  For convenience, we make the following notation 
$$\Mcd{N} = \{\widehat{\theta}(1)\mid \widehat{\theta} \in \Min{N}\}.$$ 
\begin{definition}
Let $N$ a normal subgroup of $G$. Then a {\it  leader $G$-character } of $N$ is a character of $N$ defined by the product $\theta \widehat{\theta}$, for some $\theta \in \irr{N}$. We denote by
$$\Lcd{N}= \{\theta(1)\widehat{\theta}(1)\mid \theta \in \irr{N}\}$$
the set of degrees of leader $G$-characters of $N$.
\end{definition}
Obverve that $|\Mcd{N}|$ and $|\Lcd{N}|$ may be different interger numbers (see Example 3.22). Moreover, by Corollary~\ref{relations}$(B_i)$, degrees of the leader $G$-characters of $N$ are determined by the character table of $G$.
\bigskip

The next theorem is a variation of Taketa's theorem for $G$-invariant $nMI$-subgroups. 

\begin{theorem} \label{normal} Let $N$ be a $G$-invariant $nMI$-subgroup of a group $G$ and let $1=f_1 < f_2 < \cdots < f_s$  be the distinct elements of $\Lcd{N}$. Then
$$\Gamma^{i+1}_G(N) \subseteq \ker(\widehat{\theta_i})$$
with $\theta_i(1)\widehat{\theta_i}(1) = f_i$ and $\theta_i \in \irr{N}$, for $i \in \{1, \ldots, s\}$. In particular, $N$ is a hypercentral subgroup of $G$ and $\textbf{l}_G(N) \leq s=|\Lcd{N}|$.
\end{theorem}
 \begin{proof} We prove the theorem, using induction on $|G|$.  Let $N_0$ be any normal subgroup of $G$ contained in $N$. We claim that $N/N_0$ satisfies the hypothesis of the theorem. Let $\theta\in \irr{N/N_0}$, then there exists a $G$-character pair with  respect to $\theta$, say $(H_{\theta}, \lambda_{\theta})$. Note that $\lambda_{\theta}\in {\rm Lin}_G(H_{\theta})$ and $\lambda_{\theta}^N=m\widehat\theta$, for some integer $m$. Hence, 
 $$\ker(\widehat{\theta})=\ker (\lambda_{\theta}^N)=\bigcap\limits_{x \in N}(\ker(\lambda_{\theta}))^x= \bigcap\limits_{x \in N}\ker(\lambda_{\theta}^x)= \ker(\lambda_{\theta}).$$
Then $N_0\leq \ker(\widehat{\theta})=\ker(\lambda_{\theta})\leq H_{\theta}$, which means that $\lambda_{\theta}\in \Lin{H_{\theta}/N_0}$ and so our claim is proved. 
  
 Let $\theta \in \irr{N}$. If $\theta(1)\widehat{\theta}(1)=1$, then $\theta \in \Lin{N}$ and trivially $\Gamma^2_G(N) \subseteq \ker(\theta)=\ker(\widehat\theta)$. Suppose that $\theta(1)\widehat{\theta}(1) \not = 1$. Suppose that $\ker(\widehat{\theta})\not =1$ and we work on $\bar{G}=G/\ker(\widehat{\theta})$ by adopting bar for the subgroups. If $\theta(1)\widehat{\theta}(1)=f_i$, for some $2\leq i\leq s$, then  $\theta(1)\widehat{\theta}(1)$ is $j$-th smallest degree of ${\rm Lcd}_{\bar{G}}(\bar{N})$, for some $j\leq i$.  By induction, we have 
$$\Gamma^{i+1}_G(N)\ker(\widehat{\theta})/\ker(\widehat{\theta})\leq \Gamma^{j+1}_G(N)\ker(\widehat{\theta})/\ker(\widehat{\theta})=\Gamma^{j+1}_{\bar{G}}(\bar{N})\leq \ker(\widehat{\theta})/\ker(\widehat{\theta})=1.$$
Therefore, $\Gamma^{i+1}_G(N) \subseteq \ker(\widehat{\theta_i})$ as wanted. 

 We may assume $\widehat{\theta}$ is faithful.  In this case, we claim that $H_{\theta}={\bf Z}(G)\cap N$.  Since $\ker(\lambda_{\theta})=\ker(\lambda_{\theta}^N)=\ker(\widehat\theta)=1$, then $$[H_{\theta}, G]=\bigcap\limits_{\lambda\in {\rm Lin}_G(H_{\theta})} \ker(\lambda)\leq\ker(\lambda_{\theta})=1.$$ 
Thus, $H_{\theta}\leq {\bf Z}(G)\cap N$.  On the other hand, $\widehat{\theta}|_{{\bf Z}(G)\cap N}=m\beta$, for some $\beta \in \irr{{\bf Z}(G)\cap N}$.  If $x \in {\bf Z}(G)\cap N \setminus H_{\theta}$, then $\lambda_{\theta}^N(x)=0=\widehat{\theta}(x)=m \beta(x)$, which means that $\beta(x)=0$ and  this is not possible. Therefore, ${\bf Z}(G)\cap N = H_{\theta}$.  Moreover, $$m\widehat{\theta}(1)=\lambda_{\theta}^N(1)=|N:{\bf Z}(G)\cap N|$$ and $m=\theta(1)$. Adobting tilde for convention $\tilde{G}=G/({\bf Z}(G)\cap N)$, we have $\theta(1)\widehat{\theta}(1)=|\tilde N|$. On the other side, 
$$|\tilde N|=\sum\limits_{\chi\in \irr{\tilde N}}\chi^2(1)= \sum\limits_{\widehat{\chi}\in {\rm Min}_{\tilde G}(\tilde N)}\chi(1)\widehat{\chi}(1).$$
This implies that  $\theta(1)\widehat{\theta}(1)> \chi(1)\widehat{\chi}(1)$ for all $\chi\in \irr{\tilde N}$. Let  $t=|{\rm Lcd}_{\tilde G}(\tilde N)|$, it follows that $\theta(1)\widehat{\theta}(1)$ is $i$-th smallest number in $\Lcd N$, for some  integer $i\geq t+1$. By induction, we have  $$({\bf Z}(G)\cap N)\Gamma^{t+1}_G(N)/{\bf Z}(G)\cap N=\Gamma^{t+1}_{\tilde G}(\tilde N)\leq \bigcap\limits_{\widehat{\chi}\in {\rm Min}_{\tilde G}(\tilde N)}\ker(\widehat{\chi})=({\bf Z}(G)\cap N)/({\bf Z}(G)\cap N).$$   Hence, $\Gamma^{t+2}_{G}(N)=1=\ker\widehat{\theta}$  as wanted. So the proof is complete. 
\end{proof}

\begin{corollary} \label{invariant} If $G$ is a $nMI$-group and $1=\chi_{1}(1) < \chi_2(1) < \cdots < \chi_n(1)$ are the distinct degrees of irreducible characters of $G$, then
$$\Gamma_{i+1}(G) \leq \ker(\chi_i),$$
for $i \in \{1, \ldots, n\}$. In particular, $G$ is nilpotent and the nilpotence class of $G$ is bounded by $n$. Moreover, every normal subgroup of $G$ is a $G$-invariant $nMI$-subgroup of $G$.
\end{corollary}

\begin{proof}  To prove this corollary we just need  to replace $N$ by $G$ in Theorem \ref{normal}. Clearly, in this case $\widehat{\chi}=\chi$ for $\chi\in \irr G$ and $|{\rm Lcd_G(G)}|=|{\rm cd}(G)|$. Thus the first  part of the theorem is done. 

Let $N$ be  a normal subgroup of $G$,  we claim $N$ is a $G$-invariant $nMI$-group of $G$. Let $\theta \in \irr{N}$ and $\chi \in \irr{G|\theta}$.  By hypothesis, there exists $H \trianglelefteq G$ and a linear $G$-invariant character  $\varphi \in \irr{H}$ such that $\varphi^G =m \chi$ and $\chi_H = m \varphi$, for some positive integer $m$. It follows that $\chi \in \irr{N|\varphi}$. 

Firstly, we assume $H \subseteq N$. If $\theta_1 \in \irr{N|\varphi}$, then $[\theta_1^G,\chi] \neq 1$ and $[\chi_N,\theta_1] \neq 1$, thus $\widehat{\theta_1}=\widehat{\theta}$. Then, $\varphi^N = s \widehat{\theta}$ and $\theta_H=s \varphi$, for some positive integer $s$, and the proposition is proved. Now, assume $H$ is not contained in $N$. We have $\chi_{HN}=k \widehat{\lambda}$, with $\lambda \in \irr{HN|\varphi}$ and some integer $k \geq 0$. Hence, $\lambda_H=r\varphi$ and $\varphi^{HN}=r \widehat{\lambda}$, for some positive integer $r$. Moreover,  $\chi_N=e \widehat{\theta}=k \widehat{\lambda}_N$, for some integer $e \geq 0$. Then,
$$\chi_{H \cap N}=e \widehat{\theta}_{H \cap N}=k \widehat{\lambda}_{H \cap N} = k l \beta,$$
for some positive integer $l$. Therefore, $\theta \in \irr{N|\beta}$ and $\theta_{H \cap N} = s \beta$, for some $\beta\in \irr{H\cap N}$  and some integer $s \geq 0$ . Moreover, by Frobenius reciprocity, we have
$$\beta^N=(\varphi_{H \cap N})^N=(\varphi^{HN})_N=r \widehat{\lambda}_N= s\widehat{\theta},$$
and the  proof is complete.


\end{proof}

\begin{example}
Regarding the Corollary~\ref{invariant}, we note that not all nilpotent groups verify that the class of nilpotency is bounded by the number of degrees of irreducible characters. For instance, using GAP (see \cite{GAP}), we can check that the group $G=SmallGroup(32,9)$ has nilpotency class 3 and only has two irreducible character degrees. 
\end{example}

\begin{example}
It is also not difficult to find groups that are $nMI$-groups and we can check easily this property from the character table. For example, we consider in GAP the character table of the group $G=SmallGroup(32,8)$, which has 11 irreducible characters. For every irreducible character $\chi_i \in \irr{G}$, for $i=\{1,\ldots,11\}$, there exists a normal subgroup $N_i$ of $G$ and a linear character $G$-invariant $\lambda_i \in \irr{N_i}$ such that $\irr{G|\lambda_i}= \{\chi_i\}$. Indeed, for linear characters $\chi_i$ of $G$, for $i = 1, \ldots, 8$, this fact is trivial considering $N_i=G$. The group $G$ has also two irreducible characters, $\chi_9$ and $\chi_{10}$,  of degree 2. Both characters are, respectively, the unique irreducible constituent determined by induction in $G$ of certain irreducible character $\lambda_i$, for $i \in \{9,10\}$, of the normal subgroup $$N=N_9=N_{10}= \bigcap\limits_{i=1}^4 \ker(\chi_i).$$  We can check that $\lambda_i$ is a linear $G$-invariant character of $N$ taking acount the relations of Corollary~\ref{relations}($B_i$) obtained from considering the corresponding $G$-character table of $N$ associated to $\chi_i$, for $i \in \{9,10\}$. Finally, the irreducible character $\chi_{11}$ of $G$ has degree $4$. For this character, we consider the normal subgroup $$N_{11}=\bigcap\limits_{i=1}^4 \ker(\chi_i) \bigcap  \ker(\chi_{10}).$$ We have that $\chi_{11}$ is the unique irreducible constituent determined by induction in $G$ of an irreducible character $\lambda_{11}$ of $N_{11}$. Using the corresponding $G$-character table of $N_{11}$ associated to $\chi_{11}$, and the relation Corollary~\ref{relations}$(B_i)$, we can also check that $\lambda_{11}$ is a linear $G$-invariant character of $N_{11}$. Therefore, $(N_i,\lambda_i)$ is a linear character pair with  respect to $\chi_i$, for $i \in \{1, \ldots, 11\}$, and thus $G=SmallGroup(32,8)$ is an $nMI$-group. In relation to Corollary~\ref{invariant}, the nilpotency class of $G$ is exactly equal to the number of irreducible character degrees of $G$ (that is equal to $3$).

\end{example}

\begin{example} 
The bound in the Theorem~\ref{normal} is sharp. In fact,  looking  at group $G=Q_{16}$ and its normal subgroup $N=Q_8$, we can similarly check from the character table of $G$ that $N$ is a $G$-invariant $nMI$-subgroup of $G$ but $G$ is not $nMI$-group. In this case, we have that $\Gamma^4_G(N)=1$ and $\Gamma^i_G(N)\not =1$ for $i<4$. Therefore, the hypercentral $G$-length of $N$ is $\textbf{l}_G(N)=3$.  
Moreover, it is worth mentioning that $2=|\Mcd N|=|\{1,2\}|< |{\rm Lcd_G(N)}|=|\{1,2,4\}|=3$ and so the $G$-hypercentral lenght of $N$ is not bounded by $|\Mcd N|$. 

Moreover, we can check that $N$ is not $nMI$-group since $N$ has an irreducible character $\theta$ such that $\theta(1)=2$ and it does not exist a linear pair $(H_{\theta}, \lambda_{\theta})$ respect to $\theta$  such that $\theta$ and $\lambda_{\theta}$ are fully ramified.
\end{example}

{\bf Acknowledgment.} 

The authors would like to thank the referee for his or her comments which led to the improvement of this work. This work was done during a visit of the first author at   Universitat Politècnica de València(UPV) and Universitat de València(UV). She wishes to thank them 
for their hospitality. The results in this paper are part of the third author's Ph.D. thesis, and he acknowledges the support of Instituto Tecnológico de Santo Domingo (INTEC) and Universidad Autónoma de Santo Domingo (UASD), Santo Domingo, (Dominican Republic).


\end{document}